\documentclass[reqno]{amsart}
\usepackage{amssymb,amsmath,amsthm}
\usepackage{color}
\usepackage{enumerate}
\usepackage{fullpage}
\usepackage{url}
\usepackage[Symbol]{upgreek}
\usepackage[unicode,pdfusetitle]{hyperref}
\usepackage[mathscr]{euscript}
\usepackage{mathtools}

\usepackage{lipsum}
\usepackage{xcolor}

\usepackage{cleveref}

\usepackage{tikz}
\usetikzlibrary{decorations.pathmorphing, patterns,shapes}
\usepackage{tikz-cd}

\newtheorem{theorem}{Theorem}
\newtheorem*{theorem*}{Theorem}
\newtheorem{question}[theorem]{Question}
\newtheorem*{question*}{Question}

\newtheorem*{conjecture*}{Conjecture}

\newtheorem*{convention*}{Convention}

\newtheorem*{assumption*}{Assumption}
\newtheorem{corollary}[theorem]{Corollary}
\newtheorem*{corollary*}{Corollary}
\newtheorem*{remark*}{Remark}
\newtheorem{proposition}[theorem]{Proposition}
\newtheorem*{proposition*}{Proposition}

\newtheorem*{lemma*}{Lemma}

\newtheorem*{fact*}{Fact}

\theoremstyle{definition}

\newtheorem*{definition*}{Definition}

\newtheorem{observation}[theorem]{Observation}
\newtheorem*{example*}{Example}
\newtheorem{remark}[theorem]{Remark}

\numberwithin{theorem}{section}
\numberwithin{equation}{section}

\newcommand\blfootnote[1]{%
  \begingroup
  \renewcommand\thefootnote{}\footnote{#1}%
  \addtocounter{footnote}{-1}%
  \endgroup
}

\DeclareMathOperator{\res}{res}
\DeclareMathOperator{\ac}{ac}

\DeclareMathOperator{\Th}{Th}

\newcommand{\N}{\mathbb{N}}

\newcommand{\Q}{\mathbb{Q}}

\newcommand{\Z}{\mathbb{Z}}

\newcommand{\cL}{\mathcal L}

\newcommand{\cO}{\mathcal O}

\newcommand{\sv}{{\bar s}}

\renewcommand{\sf}{\mathbf{f}}
\newcommand{\sr}{\mathbf{r}}
\renewcommand{\sv}{\mathbf{v}}

\newcommand{\ul}[1]{\underline{#1}}
\newcommand{\ol}[1]{\overline{#1}}

\newcommand{\val}{\mathrm{val}}

\renewcommand{\geq}{\geqslant}

\renewcommand{\epsilon}{\varepsilon}
\renewcommand{\phi}{\varphi}

\DeclareFontFamily{OMS}{smallo}{}
\DeclareFontShape{OMS}{smallo}{m}{n}{<->s*[.65]cmsy10}{}
\DeclareSymbolFont{smallo@m}{OMS}{smallo}{m}{n}
\DeclareMathSymbol{\smallo}{\mathord}{smallo@m}{79}

\DeclareFontFamily{U}{fsy}{}
\DeclareFontShape{U}{fsy}{m}{n}{<->s*[.9]psyr}{}
\DeclareSymbolFont{der@m}{U}{fsy}{m}{n}
\DeclareMathSymbol{\der}{\mathord}{der@m}{182}
\DeclareSymbolFont{der@m}{U}{fsy}{m}{n}
\DeclareMathSymbol{\derdelta}{\mathord}{der@m}{100}

\newcommand{\Lor}[0]{\mathcal{L}_{\mathrm{or}}}
\newcommand{\Lovf}[0]{\mathcal{L}_{\mathrm{ovf}}}
\newcommand{\Lr}[0]{\mathcal{L}_{\mathrm{r}}}
\newcommand{\Log}[0]{\mathcal{L}_{\mathrm{og}}}
\newcommand{\Lovfac}[0]{\mathcal{L}_{\mathrm{ovf}}^{\mathrm{ac}}}

\title{Ordered henselian valued fields: definability and Borel sets}
\author{Lothar Sebastian Krapp {\rm and} Floris Vermeulen}
\begin{document}

\begin{abstract}
	We firstly show that due to their resplendency ordered henselian valued fields admit relative field quantifier elimination in the Denef--Pas language expanded by linear orders in the field and residue field sort. Secondly, we deduce from a dimensionality reduction theorem that any set definable over an ordered henselian valued field is a Borel set with respect to the order topology. Our results are contextualised within Shelah's classification conjecture of NIP fields and its connections to the study of definable henselian valuations and the Fundamental Theorem of Statistical Learning.
\end{abstract}
\maketitle

\begin{NoHyper}
\blfootnote{Math Subject Classification (2020): Primary 03C64, 03C10; Secondary 12J10, 12L12, 12J25, 12J15. Keywords: relative quantifier elimination, ordered fields, henselian valued fields, NIP, measurable.}
\end{NoHyper}
\vspace{-1cm}
\section{Introduction}

	Ordered algebraic structures can be analysed with view to their model theoretic and topological properties. From the viewpoint of Model Theory, these structures are  endowed with algebraic operations and order relations, i.e.\ interpretations of suitable first-order languages. In turn, the study of first-order definable subsets becomes of interest in order to gain an understanding of the expressive power of the corresponding languages. In the presence of a \emph{linear} order, algebraic structures naturally obtain the order topology generated by open intervals, making them topological spaces.

    In this note, we focus on the model theoretic and topological examination of (linearly) ordered fields equipped with a henselian valuation. Our first main result (\Cref{main1}) establishes relative field quantifier elimination for these structures in a three-sorted language (field, residue field and value group sort). Here, the field and residue field are considered in the language $\Lor=\{+,-,\cdot,0,1,<\}$ of ordered rings and the value group in the language $\Log=\{+,-,0,<\}$ of ordered groups. Besides the valuation and the residue map, our three-sorted language requires an angular component map, making it an expansion of the Denef--Pas language (see \cite[Definition~6.5]{hils}) by orders in the field and residue field sort. A special kind of ordered henselian valued fields always admitting angular component maps is pointed out in \Cref{prop:arcac}, namely almost real closed fields, i.e.\ ordered fields admitting a henselian valuation with real closed residue field. 
    
    A weak version of relative field quantifier elimination was already obtained in \cite[Lemma~4.1]{dittmann}. While there the existence of an angular component map is not required, the weak version of quantifier elimination comes at the cost of losing control over parameters. More precisely, given a formula with constant field sort symbols, the equivalent field-quantifier free formula may contain constant residue field and value group sort symbols that the initial formula did not contain. In \cite[Theorem~4.2]{dittmann}, the weak quantifier elimination result is eventually used in order to show that both the \emph{ordered} residue field and the ordered value group are stably embedded into the ordered henselian valued field, and that they are orthogonal. This, in turn, is the main ingredient of the main result of \cite{dittmann}, which establishes that any henselian valuation that is $\Lor$-definable over an ordered field $K$ is already $\Lr$-definable over $K$.\footnote{Saying that a valuation is definable over $K$ means that its corresponding valuation ring is a definable subset of $K$.} Based on these findings, the following question is raised in \cite[Question~2.2]{krapphabil}.

    \begin{question}\label{qu:lorlr}
        Is every henselian valuation on an ordered field that is parameter-free $\Lor$-definable already parameter-free $\Lr$-definable?
    \end{question}

    We provide in \Cref{cor:parfreeemb} a helpful tool for a potential approach towards \Cref{qu:lorlr}. Namely, we improve \cite[Theorem~4.2]{dittmann} for the parameter-free case by showing that parameter-free definable sets in the field induce parameter-free definable sets in residue field and value group. We thus provide a positive answer to \cite[Question~2.3]{krapphabil}.

    One recent motivation to study the definability of henselian valuations (see \cite[page 135]{fehm} for a survey) stems from Shelah's conjecture on the classification of infinite NIP fields: \emph{every infinite NIP field is separably closed, real closed, or admits a non-trivial definable henselian valuation.}\footnote{This conjecture goes back to \cite[Conjecture~5.34 (c)]{shelah} and has been reformulated since; cf.\ e.g.\ \cite[Conjecture~1.9]{johnson} and \cite[Conjecture~1.1]{anscombe}.}
    In \cite{krapp}, Shelah's conjecture is specialised to ordered fields (in the strongly dependent context). As an application of \Cref{main1}, we show in \Cref{prop:arcniporder} that any almost real closed field is NIP as an ordered field for any possible order. Thus, Shelah's conjecture implies that any (formally) real field that is NIP as a pure field is also NIP as an ordered field for any possible order. 
    In a similar vein as \cite[Section~5.5]{krapp}, we present in \Cref{prop:equivalentconjecturesa} 
    an equivalent version of Shelah's conjecture specialised to real fields.

    The study of NIP ordered fields opens up interesting connections to Statistical Learning Theory, the foundational framework of Machine Learning (see \cite[Section~1]{krappwirth}). Within this framework, the general task is to find a function that best describes a general relationship between inputs and (binary) outputs based on only finitely many known input-output samples. This leads to the concept of \emph{probably approximately correct} learning (see \cite[Section~2]{krappwirth}), which constitutes that the task has been completed successfully if after knowing a random finite selection of input-output samples a function can be chosen that with a high probability (i.e.\ probably) only makes few errors (i.e.\ is approximately correct) in describing the input-output relationship. 
    This can be achieved if the collection of functions to choose from has a finite Vapnik--Chervonenkis (VC) dimension and is ``well-behaved'' with respect to several (Borel) measure theoretic features (see \cite[Section~3]{krappwirth}). Finite VC dimensions and measure theoretic well-behavedness are the main conditions for the applicability of the Fundamental Theorem of Statistical Learning (see \cite[Theorem~3.3]{krappwirth}). In this context, from a model theoretic perspective collections of \emph{definable} functions are of particular interest (see \cite[Section~4]{krappwirth}). If the ambient structure is ``tame'' enough, then the two main conditions for the Fundamental Theorem of Statistical Learning are satisfied. Namely, any NIP structure ensures finite VC dimensions for definable function classes, and any o-minimal expansion of the reals guarantees all necessary measurability conditions (see \cite[Theorem~4.7]{krappwirth}). Note that o-minimality implies NIP and is therefore a stronger tameness condition. The question emerges whether the property NIP is already enough to ensure that definable sets are measure theoretically well-behaved. More precisely, the issue arises whether every set definable in an NIP ordered field is already a Borel set. In light of Shelah's conjecture, this prompted \cite[Question~4.3]{krappwirthvermeil} as follows.
    \begin{question}\label{qu:arcmeasure}
        Is there an almost real closed field $K$ that defines a set that is not Borel with respect to the order topology on $K$?
    \end{question}
    We provide in \Cref{prop:arc.borel} a negative answer to \Cref{qu:arcmeasure}. \Cref{prop:arc.borel}, in turn, is derived from the second main result of this note (\Cref{thm:secondmain}), which exploits dimension theory of definable sets in h-minimal structures.

\section{Preliminaries}

\subsection{Notations and conventions}
    
    Tuples are denoted by underlining, e.g.\ $\underline{x}$ for a tuple $(x_1,\ldots,x_\ell)$, where the length $\ell$ is only specified if relevant. All orders we consider are linear. 
    
    Throughout this note, $K=(K,+,-,0,1)$ denotes a field (as an $\Lr$-structure), $v$ is a henselian valuation on $K$, $vK=(vK,+,-,0,<)$ denotes the value group (as an $\Log$-structure) and $Kv=(Kv,+,-,0,1)$ the residue field (as an $\Lr$-structure) under $v$. We express the valuation ring of $v$ by $\mathcal{O}_v$, its maximal ideal by $\mathcal{M}_v$ and the residue $a+\mathcal{M}_v$ for $a\in \mathcal{O}_v$ by $\ol{a}$. If $<$ is an order on $K$, then $(K,<)$ denotes the $\Lor$-structure of an ordered field. 
    By abuse of notation, we write $(K,<,v)$ for the structure $(K,+,-,\cdot,<,\mathcal{O})$ in the language $\Lovf$ of ordered valued fields, where $\mathcal{O}$ is a unary relation symbol defined as $\mathcal{O}(x):\Leftrightarrow v(x)\geq 0$. 
    Since $v$ is henselian and thus convex (meaning that $\mathcal{O}_v$ is a convex subset of $K$, i.e.\ for any $a,b\in \mathcal{O}_v$ and any $c\in K$ with $a<c<b$, also $c\in \mathcal{O}_v$), the residue field $Kv$ can be endowed with the induced order, making it an $\Lor$-structure $(Kv,<)$ of an ordered field (see \cite[Section~4.3]{engler} for details).

    We say that $K$ is \emph{(formally) real} if it admits an order $<$ making $(K,<)$ an ordered field. A real field is considered as an $\Lr$-structure whereas an ordered field is considered as an $\Lor$-structure.

\subsection{Almost real closed fields}
    A field $K$ is called \emph{almost real closed} if it admits a henselian valuation $v$ such that $Kv$ is real closed. Almost real closed exhibit several desirable model and valuation theoretic properties (see \cite{delon} for an extensive study). For instance, all convex valuations on an almost real closed field are henselian (see \cite[Proposition~2.2\,(iii)]{delon}), and the set of all henselian valuations is linearly ordered by the coarsening relation (see \cite[Proposition~2.1\,(i)]{delon}). Thus, the \emph{canonical henselian valuation} $v_K$ on an almost real closed field $K$ is, at the same time, the finest convex valuation with respect to any order on $K$. Note that $Kv_K$ is real closed due to \cite[Proposition~2.1\,(iv)]{delon}. Almost real closed fields are naturally endowed with an order topology. We establish in \Cref{obs:alltopsame} below that this topology is independent of the order and therefore only refer to it as \emph{the} order topolgy on $K$.

    \begin{observation}\label{obs:alltopsame}
        Let $K$ be an almost real closed field. Then any henselian valuation on $K$ and any order on $K$ induce the same topology.
    \end{observation}

    \begin{proof}
        Recall that the topology induced by a non-trivial convex valuation on an ordered field coincides with the order topology (see e.g.\ \cite[Section 7.63]{alling}). The observation follows from the fact that henselian valuations on almost real closed fields are exactly the valuations that are convex with respect to some (and therefore any) order.
    \end{proof}

\section{Relative quantifier elimination}

    Following the terminology of \cite[Section~4]{dittmann}, we consider the three-sorted language of ordered valued fields with angular component map
    $$\Lovfac=(\Lor,\Lor,\Log;\overline{\ \cdot\ },v,\ac).$$
    The correpsonding sorts are $\sf$ (field sort), $\sr$ (residue field sort) and $\sv$ (value group sort), where the unary function symbols have sorts $\overline{\ \cdot\ }\colon \sf\to\sr$, $v\colon \sf\to\sv$ and $\ac\colon\sf\to\sr$.
    Let $(K,<)$ be an ordered field with henselian valuation $v$. We say that an angular component map\footnote{See \cite[Section 5.4]{dries} for details on angular component maps.}  $\ac\colon K\to Kv$ is \emph{compatible} (with the order $<$ on $K$) if for $x\in K$ we have $x>0$ if and only if $\ac(x)>0$. If $K$ admits a compatible angular component map $\ac$, then the $\Lovfac$-structure associated to $K$ is given by
    $$\mathcal{K}=((K,<),(Kv,<),vK;\overline{\ \cdot\ },v,\ac).$$
    By convention, we set $v(0)=0$ and $\overline{a}=0$ for any $a\in K$ with $v(a)<0$.

    \begin{observation}\label{lem:saturatedcompatibleac}
        Any $\aleph_1$-saturated ordered henselian valued field $(K,<,v)$ admits a compatible angular component map.
    \end{observation}

    Let $B$ be an abelian group and let $A\subseteq B$ be a subgroup. Recall that $A$ is \emph{pure} in $B$ if for every positive integer $n$ we have $A\cap nB = nA$.
    
    \begin{proof}
        This is similar to e.g.\ ~\cite[Lemma~7.9]{dries}, but replacing the usage of $K^\times$ by $K^\times_{> 0}$.
        Let $U = \cO_v^\times \cap K_{>0}$ be the set of positive units.
        Consider the exact sequence
        \[
        1\to U \to K_{>0}^\times \to vK \to 0.
        \]
        Since $vK$ is torsion-free, $U$ is a pure subgroup of $K_{>0}$.
        Hence~\cite[Corollary~7.8]{dries} shows that this sequence splits using a map $\sigma\colon K^\times_{>0}\to U$. 
        Then defining 
        \[
        \ac(x) = \begin{cases} \overline{\sigma(x)}, & x>0, \\
            -\overline{\sigma(-x)}, & x < 0 \end{cases}
        \]
        yields a compatible angular component.
    \end{proof}

    Without saturation \Cref{lem:saturatedcompatibleac} remains true for almost real closed fields.
    
    \begin{proposition}\label{prop:arcac}
        Let $K$ be an almost real closed field endowed with an order $<$.
        Then there exists a compatible angular component map $\ac\colon K\to Kv_K$.
    \end{proposition}

    \begin{proof}
        By~\cite[Lemma 3.1.1]{NSV}, there exists a section $s\colon v_KK\to K_{>0}$ of the valuation.
        Now simply define $\ac(x) = \res(x/s(v_K(x)))$.
    \end{proof}

    For the first main result \Cref{main1} and its corollaries, we fix the following notation and terminology oriented at \cite[Section~5.6]{dries}. Let $\underline{x}=(x_1,\ldots,x_\ell)$, $\underline{y}=(y_1,\ldots,y_m)$ and $\underline{z}=(z_1,\ldots,z_n)$ be tuples of distinct $\sf$-, $\sr$- and $\sv$-variables, respectively. Given an $\Lor$-formula $\psi'(u_1,\ldots,u_k,\underline{y})$ (only using free and bound $\sr$-variables), an $\Log$-formula $\theta'(v_1,\ldots,v_k,\underline{z})$ (only using free and bound $\sv$-variables) and polynomials $q_1(\ul{x}),\ldots,q_k(\ul{x})\in\Z[\ul{x}]$, we call 
    $$\psi(\ul{x},\ul{y}) := \psi'(\ac(q_1(\ul{x})),\ldots,\ac(q_k(\ul{x})),\ul{y})$$
    a \emph{special $\sr$-formula} and 
    $$\theta(\ul{x},\ul{z}) := \theta'(v(q_1(\ul{x})),\ldots,v(q_k(\ul{x})),\ul{z})$$
    a \emph{special $\sv$-formula}.

    \begin{theorem}[Relative quantifier elimination]\label{main1}
        Let $\mathcal{K}=((K,<),(Kv,<),vK;\overline{\ \cdot\ },v,\ac)$ be the $\Lovfac$-structure of an ordered henselian valued field with a compatible angular component map. Then the theory of $\mathcal{K}$ in $\Lovfac$ eliminates field quantifiers. More precisely, for any $\Lovfac$-formula $\varphi(\ul{x},\ul{y},\ul{z})$ there are some $N\in\N$, special $\sr$-formulas $\psi_1(\ul{x},\ul{y}),\ldots,\psi_N(\ul{x},\ul{y})$ and special $\sv$-formulas $\theta_1(\ul{x},\ul{z}),\ldots,\theta_N(\ul{x},\ul{z})$ such that $\varphi(\ul{x},\ul{y},\ul{z})$ is equivalent over $\mathcal{K}$ to        $$(\psi_1(\ul{x},\ul{y})\wedge\theta_1(\ul{x},\ul{z}))\vee\ldots\vee(\psi_N(\ul{x},\ul{y})\wedge\theta_N(\ul{x},\ul{z})).$$
    \end{theorem}

    \begin{proof}
        By Pas' theorem~\cite[Theorem~4.1]{Pas}, there is relative quantifier elimination in the smaller language 
        \[
        \cL_{\mathrm{vf}}^{\ac} = (\cL_{\mathrm{r}}, \cL_{\mathrm{r}}, \cL_{\mathrm{og}}; \overline{\ \cdot\ }, v, \ac),
        \]
        where we omit the order on both the residue field and the valued field.

        Now, in the terminology of~\cite[Definition~A.7]{Silvain} the residue field is a closed sort, and hence~\cite[Proposition~A.9]{Silvain} shows that quantifier elimination is resplendent with respect to the residue field.
        This means that for every expansion of the language $\cL'\supseteq \cL_{\mathrm{r}}$ on the residue field, there is still relative quantifier elimination in
        \[
        (\cL_{\mathrm{r}}, \cL', \cL_{\mathrm{og}}; \overline{\ \cdot\ }, v, \ac).
        \]
        In particular this applies for $\cL' = \cL_{\mathrm{r}}\cup\{<\}.$

        To conclude, note that the order $<$ on $K$ is quantifier-free definable using the order on $Kv$.
        Indeed, since $\ac$ is compatible with $<$ we have for $x\in K$ that $x>0$ if and only if $\ac(x)>0$.
    \end{proof}

    As a corollary, we recover stable embeddedness and orthogonality of the ordered residue field and the value group with control over parameters (cf.\ \cite[Theorem~4.2]{dittmann}).

    \begin{corollary}\label{cor:stablembed}
        Let $(K,<,v)$ be an ordered henselian valued field, and let $m,n\in\N$. 
        Then the following hold.
        \begin{enumerate}
            \item \label{it:rf} (Residue field) Suppose that $(K,<,v)$ admits a compatible angular component $\ac$ and let $A\subseteq K$ be a subring.
            Then every $\Lovfac(A)$-definable set $X\subseteq (Kv)^n$ is already $\Lor(\ac(A))$-definable.

            \item \label{it:vg} (Value group) Let $A\subseteq K$ be a subring.
            Then every $\Lovf(A)$-definable set $Y\subseteq (vK)^m$ is already $\Log(v(A))$-definable.
        
            \item \label{it:ortho} (Orthogonality) Suppose that $(K,<,v)$ admits a compatible angular component $\ac$ and let $A\subseteq K$ be a parameter set.
            Then every $\Lovfac(A)$-definable subset of $(Kv)^n\times (vK)^m$ is a finite union of rectangles $Y\times Z$, where $Y\subseteq (Kv)^n$ and $Z\subseteq (vK)^m$ are $\Lovfac(A)$-definable.
        \end{enumerate}
    \end{corollary}
        
    \begin{proof}
        \Cref{main1} immediately implies (\ref{it:rf}) and (\ref{it:ortho}).
        Also (\ref{it:vg}) follows by working in an $\aleph_1$-saturated elementary extension, where a compatible angular component map exists by \Cref{lem:saturatedcompatibleac}.
    \end{proof}

    Note that \Cref{cor:stablembed}~(\ref{it:vg}) holds even without the existence of a compatible angular component map, as this map is only needed to map parameters of sort $\sf$ to parameters of sort $\sr$. In the presence of a compatible angular component map, a similar conclusion as in \Cref{cor:stablembed}~(\ref{it:rf}) can be made for the value group. By application of \Cref{lem:saturatedcompatibleac}, we also obtain that the angular component map is not needed at all in the setting of parameter-free definability.

    \begin{corollary}\label{cor:parfreeemb}
        Let $(K,<,v)$ be an ordered henselian valued field, and let $X\subseteq K^n$ be parameter-free $\cL_{\mathrm{ovf}}$-definable. Then $\overline{X}$ is parameter-free $\cL_{\mathrm{or}}$-definable in $(Kv)^n$, and $v(X)$ is parameter-free $\cL_{\mathrm{oag}}$-definable in $(vK)^n$. 
    \end{corollary}

    \Cref{cor:parfreeemb} shows that sets definable without parameters in the field induce sets in the residue field and the value group that are also definable without parameters. We point out that this provides positive answer to \cite[Question~2.3]{krapphabil}, which was actually only stated for the case $m=n=1$.
    
    \begin{remark}
        The language $\Lovf'$ from~\cite[Section~4]{dittmann}, which is the reduct of $\Lovfac$ obtained by removing $\ac$, is too weak to obtain relative quantifier elimination in the form of \Cref{main1}. Indeed, consider the ordered valued field $\Q(\!(t)\!)$ with the usual $t$-adic valuation $v_t$. Then $t^2$ and $2t^2$ satisfy exactly the same valued field quantifier-free formulas in the language $\Lovf'$, but one is a square while the other is not. Of course, this is false in the language $\Lovfac$, as $\ac(t^2) = 1$ is a square in $Kv$ while $\ac(2t^2) = 2$ is not.
    \end{remark}

    As another application of \Cref{main1}, we now turn to the classification conjecture of NIP real fields.\footnote{We say that a structure is NIP if its complete theory is NIP, i.e.\ does not have the independence property. The independency property is treated as a black box here, as we do not make use of its formal definition. We refer to relevant literature wherever specific results are needed. See \cite{simon} for an introduction to NIP theories.} Recall from the introduction that the general classification conjecture states as follows: \emph{every infinite NIP field is separably closed, real closed, or admits a non-trivial definable henselian valuation.} 
    Since real fiels are always infinite and cannot be separably closed, the conjecture specialises as follows: \emph{every  NIP real field is real closed, or admits a non-trivial definable henselian valuation.} In the following, we review this specialised conjecture in the context of almost real closed fields to establish an equivalent version in \Cref{prop:equivalentconjecturesa} 
    below.
    It is well-known that any almost real closed field is NIP (see \cite[Corollary~A.16]{simon}) when considered as a pure field, i.e.\ as an $\Lr$-structure.
    In \Cref{prop:arcniporder}, we present a self-contained proof that almost real closed fields are also NIP \emph{with} the order in the language. We point out that a similar account for the strongly dependent context is available in \cite[Section~4]{krapp}.
	
    \begin{proposition}\label{prop:arcniporder}
	Let $K$ be an almost real closed field and let $<$ be an order on $K$. Then $(K,<)$ is NIP as an $\Lor$-structure.
    \end{proposition}

    \begin{proof}
        Let $\varphi$ be a (partitioned) $\Lor$-formula with free variables $\ul{x}$. We show that $\varphi$ is NIP over $(K,<)$.
        By \Cref{prop:arcac}, $(K,<,v_K)$ admits a compatible angular component map. 
        We can thus apply \Cref{main1} to obtain that $\varphi(\ul{x})$ is equivalent over $\mathcal{K}=((K,<),(Kv_K,<),v_KK;v_K,\ol{\ \cdot \ },\ac)$ to a Boolean combination of special $\sr$- and $\sv$-formulas         $$(\psi_1(\ul{x})\wedge\theta_1(\ul{x}))\vee\ldots\vee(\psi_N(\ul{x})\wedge\theta_N(\ul{x})).$$
        Recall that boolean combinations of NIP formulas are NIP (see \cite[Lemma~2.9]{simon}). Using the notation from \Cref{main1}, it thus suffices to show that for any $\Lor$-formula $\psi'(u_1,\ldots,u_k)$, any $\Log$-formula $\theta'(v_1,\ldots,v_k)$ and any polynomials $q_1(\ul{x}),\ldots,q_k(\ul{x})\in\Z[\ul{x}]$, $\psi'(\ac(q_1(\ul{x})),\ldots,\ac(q_k(\ul{x})))$ and $\theta'(v(q_1(\ul{x})),\ldots,v(q_k(\ul{x})))$ are NIP over $\mathcal{K}$. But this follows immediately from the fact that any $\Lor$-formula $\psi'$ is NIP over the o-minimal structure $(Kv_K,<)$ and that any $\Log$-formula $\theta'$ is NIP over the ordered abelian group $v_KK$ (see \cite[Example~2.12 \& Theorem~A.8]{simon}).
    \end{proof}

    \begin{remark}
        In \cite{krapp}, the refinements ``dp-minimal'' and ``strongly dependent'' of NIP are considered for ordered fields. Complete characterisations of dp-minimal and strongly dependent ordered fields are given in \cite[Proposition~4.4 \& Theorem~4.12]{krapp}. Likewise, a characterisation of distal ordered almost real closed fields (cf.\ \cite[Main Theorem]{ACGZ}) may be obtained by application of \Cref{main1}, following the proof of \Cref{prop:arcniporder}.
    \end{remark}
    
    \Cref{prop:arcniporder} together with known arguments from the literature (see e.g.\ \cite[Section~5]{krapp} and the references therein) allows us to establish an equivalent version of Shelah's conjecture specialised to the classification of NIP ordered and real fields as follows.
    
        \begin{proposition}\label{prop:equivalentconjecturesa}
            The following are equivalent.
            \begin{enumerate}
                \item\label{prop:equivalentconjectures:1} Every NIP real field is either real closed or admits a non-trivial $\Lr$-definable henselian valuation.

                \item\label{prop:equivalentconjectures:3} Every NIP real field is almost real closed.
            \end{enumerate}
        \end{proposition}

        \begin{proof}
            Since every almost real closed field that is not real closed admits a non-trivial $\Lr$-definable henselian valuation (see \cite[Theorem~2.3]{fehm}), (\ref{prop:equivalentconjectures:3}) implies (\ref{prop:equivalentconjectures:1}).
            Conversely assume (\ref{prop:equivalentconjectures:1}) and let $K$ be an NIP real field. Let $v$ be the finest henselian valuation on $K$. Since $Kv$ is also an NIP real field (see \cite[Corollary~2.7]{jahnke}), by (\ref{prop:equivalentconjectures:1}) it is either real closed or admits a non-trivial henselian valuation. The latter is not possible, as a non-trivial henselian valuation on $Kv$ gives rise to a henselian strict refinement of $v$ on $K$. Thus, $Kv$ is real closed.
        \end{proof}

        We point out that several other equivalent versions of Shelah's conjecture specalised to real fields may be established by following the arguments in \cite[Section~5]{krapp}.

\section{Definable Borel sets}

    Recall from \Cref{obs:alltopsame} that any henselian valuation and any order on an almost real closed field induce the same topology, which we simply refer to as the order topology. We say that a set is Borel if it is Borel in this topology.
We show that definable sets in almost real closed fields are automatically Borel.
First we show more generally that this holds in Hensel minimal structures, similar to the o-minimal setting as in~\cite{KM, Kaiser}.

\begin{theorem}\label{thm:secondmain}
    Let $K$ be a valued field of characteristic zero, considered in some language $\cL\supseteq \cL_\val = \{+,\cdot, \cO\}$, and assume that $\Th_\cL(K)$ is $1$-h-minimal.
    Then every definable set $X\subseteq K^n$ is Borel.
\end{theorem}

For the proof we only use dimension theory of definable sets in h-minimal structures, see~\cite[Proposition 5.3.4]{CHR} and~\cite[Proposition 3.1.1(3)]{CHRV}.

\begin{proof}
    We induct on $\dim X$.
    If $\dim X = 0$ then $X$ is finite and so the result is clear.
    So assume that $\dim X = k > 0$.
    Since $\overline{X}$ is closed it is Borel, and~\cite[Proposition 3.1.1(3.f)]{CHRV} shows that $\dim(\overline{X}\setminus X) < k$. 
    Hence by induction $\overline{X}\setminus X$ is Borel, and
    we conclude that also $X = \overline{X}\setminus (\overline{X}\setminus X)$ is Borel.
\end{proof}

	\begin{corollary}\label{prop:arc.borel}
		Let $K$ be an almost real closed field and let $<$ be an order on $K$. Let $A\subseteq K^n$ be an $\Lor$-definable subset. Then $A$ is Borel.
	\end{corollary}

\begin{proof}
    If $K$ is real closed then this follows from~\cite[Lemma 11]{KM}.
    Otherwise, $v_K$ is a convex henselian valuation on $K$.
    As an ordered valued field, $K$ is $1$-h-minimal by~\cite[Theorem 6.2.1]{CHR} and~\cite[Theorem 4.1.19]{CHR}. 
    Therefore the result follows from \Cref{thm:secondmain}.
\end{proof}

    \begin{proposition}\label{prop:final}
        Assume that any of the equivalent statements of \Cref{prop:equivalentconjecturesa} (i.e.\ Shelah's conjecture specialised to real fields) holds true. Then for any NIP ordered field $(K,<)$ any $\Lor$-definable set over $(K,<)$ is Borel in the order topology.
    \end{proposition}

    \begin{proof}
        By assumption, the real field $K$ is almost real closed. 
        \Cref{prop:arc.borel} yields the conclusion.
    \end{proof}
    \Cref{prop:final} shows that Shelah's conjecture implies a positive answer to \cite[Question~1.1]{krappwirthvermeil}.
	
\subsection*{Acknowledgements}

    This work was initiated within the programme \textsl{MINT-Innovationen 2022} funded by Vector Stiftung and resumed at \textsl{ddg40 : Structures algébriques et ordonnées} at Observatoire Océanologique de Banyulssur-Mer in August 2025. We wish to thank all institutions for their hospitality.
    The author L.~S.~K.\ was supported by the \textsl{Network Platform Connecting Statistical Logic, Dynamical Systems and Optimization} of Universität Konstanz. The author F.~V.\ was supported by Humboldt Foundation.

	We thank Salma Kuhlmann for giving us initial ideas, Laura Wirth for insightful discussions, and the anonymous referee for helpful comments and questions.

\subsection*{Authorship Contribution Statement:} Both authors contributed equally to this work.

\begin{footnotesize}
	
\end{footnotesize}
	\medskip\
	
	\vfill

	\textsc{Lothar Sebastian Krapp}, Institut für Interdisziplinäre Sprachevolutionswissenschaft, Universität Zürich, Switzerland \& Fachbereich Mathematik und Statistik, Universität Konstanz, Germany\\
	\textit{E-mail address}: \texttt{sebastian.krapp2@uzh.ch}

        \medskip
	
	\textsc{Floris Vermeulen}, Fachbereich Mathematik und Informatik, Universität Münster, Germany\\
	\textit{E-mail address}: \texttt{florisvermeulen.math@gmail.com}

\end{document}